\documentclass[12pt]{amsart}
\advance\hoffset by -0.5 truecm
\usepackage{amsmath}
\usepackage{amssymb}
\usepackage{amsthm}
\newtheorem{Theorem}{Theorem}[section]
\newtheorem{Lemma}[Theorem]{Lemma}

\newtheorem{Proposition}[Theorem]{Proposition}

\newtheorem*{Definition}{Definition}
\newtheorem*{Question}{Question}

\def \dim{{\mbox {dim}}\,}

\def\V{\mbox{Var}}

\def\Z{{\mathbb Z}}
\def\R\re
\def\V{\bf V}

\def \re{{\mathbb R}}

\def \C{{\mathbb C}}

\def \V{{\bf V}}

\newcommand{\id}{\mathrm{Id}}

\newcommand{\abs}[1]{\lvert #1 \rvert}
\newcommand{\norm}[1]{\lVert #1 \rVert}

\newcommand{\mC}{\mathbb{C}}

\begin{document}
\title[Tensor tomography: progress and challenges]{Tensor tomography: progress and challenges}

\author[G.P. Paternain]{Gabriel P. Paternain}
\address{ Department of Pure Mathematics and Mathematical Statistics, University of Cambridge, Cambridge CB3 0WB, UK. {\it E-mail address: \bf \tt g.p.paternain@dpmms.cam.ac.uk}}

\author[M. Salo]{Mikko Salo}
\address{Department of Mathematics  and Statistics, University of Jyv\"askyl\"a. {\it E-mail address: \bf \tt mikko.j.salo@jyu.fi}}

\author[G. Uhlmann]{Gunther Uhlmann}
\address{Department of Mathematics, University of Washington and Fondation Sciences Math\'ematiques de Paris. {\it E-mail address: \bf \tt gunther@math.washington.edu}}




\begin{abstract} 
We survey recent progress in the problem of recovering a tensor field from its integrals along geodesics. We also propose several open problems.
\end{abstract}

\maketitle

\tableofcontents

\section{Introduction}
\label{sec:intro}

This paper surveys recent results on the integral geometry problem of recovering a tensor field from its integrals along geodesics. The most basic example of the kinds of transforms studied in this paper is the X-ray (or Radon) transform in the plane, which encodes the integrals of a function $f$ in $\re^2$ over straight lines:
$$
Rf(s,\omega) = \int_{-\infty}^{\infty} f(s\omega + t\omega^{\perp}) \,dt, \quad s \in \re, \omega \in S^1.
$$
Here $\omega^{\perp}$ is the rotation of $\omega$ by $90$ degrees counterclockwise. The properties of this transform are classical and well studied \cite{Hel}. The X-ray transform forms the basis for many imaging methods such as CT and PET in medical imaging.

A number of imaging methods involve generalizations of this transform. In seismic and ultrasound imaging one encounters ray transforms where the measurements are given by integrals over more general families of curves, often modeled as the geodesics of a Riemannian metric. Moreover, integrals of vector fields or other tensor fields instead of just integrals of functions over geodesics may arise, and these transforms are also useful in rigidity questions in differential geometry. We will give more specific examples after having defined the relevant transforms precisely.

The \emph{geodesic ray transform} acts on tensor fields on a compact, oriented Riemannian manifold $(M,g)$ with boundary of dimension $\dim(M) = n \geq 2$. We denote by $\langle \,\cdot\,,\,\cdot\,\rangle$ the $g$-inner product of tangent vectors or other tensors, and by $\abs{\,\cdot\,}$ the $g$-norm. Let $\nu$ denote the unit outer normal to $\partial M.$ We denote
by $ SM \rightarrow M$  the unit-sphere bundle
over $M$:
$$SM =\bigcup\limits_{x\in
M}S_{x},\quad S _{x}=\{ v \in T_{x}M:\left|v  
\right|_g =1\}.$$ 
The set $SM$ is a $(2n-1)$-dimensional compact manifold with boundary, which can be
written as the union $\partial(SM) =\partial
_{+}(SM) \cup
\partial _{-}(SM) $,
$$\partial _{\pm }(SM)
=\{(x, v)\in \partial(SM) ,\;\mp \,\langle \nu(x) ,v \rangle \geq 0\;\}.$$

The standard volume forms on $SM$ and $\partial(SM)$ that we will use are defined by 
$$
\begin{array}{rcl}
d\Sigma^{2n-1} &=& dV^{n}\wedge dS_{x} \\
d\Sigma^{2n-2} &=& dV^{n-1}\wedge dS_{x} \\
\end{array}
$$
where $dV^{n}$ (resp. $dV^{n-1}$) is the volume form of $M$ (resp. $\partial M$ ), and $dS_x=\displaystyle\sqrt{\det g(x)}dE_{x}$ where $dE_{x}$ is the Euclidean volume form of $S_{x}$ in $T_{x}M$.
For $(x,v)\in\partial(SM)$, let $\mu(x,v)=\abs{\langle \nu(x), v \rangle}$ and let $L^{2}_{\mu}(\partial_{+}(SM))$ be the space of functions on $\partial_{+}(SM)$ with inner product
$$(u,v)_{L^{2}_{\mu}(\partial_{+}(SM))}=\displaystyle\int_{\partial_{+}(SM)}uv\mu\,d\Sigma^{2n-2}.$$

Without loss of generality, we may assume that $(M,g)$ is embedded in $(N,g)$ where $N$ is a compact $n$-dimensional manifold without boundary.
Let $\varphi_{t}$ be the geodesic flow on $N$ 
and $X=\frac{d}{dt}\varphi_{t}|_{t=0}$ be the geodesic
vector field. If $(x,v)\in SM$, let $\gamma(t,x,v)$ be the unit speed $N$-geodesic starting from $x$ in the direction of $v$. Then 
$$
\varphi_t(x,v) = (\gamma(t,x,v), \dot{\gamma}(t,x,v)).
$$
Define the travel time $\tau:SM\rightarrow[0,\infty]$ by
$$\tau(x,v)=\inf\{t>0:\gamma(t,x,v)\in N\backslash M\}.$$
We say that $(M,g)$ is \emph{non-trapping} if $\tau(x,v)<\infty$ for all $(x,v)\in SM$.

\begin{Definition}
The \emph{geodesic ray transform} of a function $f \in C^{\infty}(SM)$ is the function 
\begin{equation*}
If(x,v)=\int\limits_{0}^{\tau(x,v)}f(\varphi_{t}(x,v)) \,dt,\quad
(x,v)\in \partial_+(SM).
\end{equation*}
\end{Definition}

Note that if the manifold $(M,g)$ is non-trapping and has strictly convex boundary, then $I:C^{\infty}(SM)\rightarrow C(\partial_{+}(SM))$, and Santal\'o's formula \cite{DPSU} implies that $I$ is also a bounded map $L^{2}(SM)\rightarrow L^{2}_{\mu}(\partial_{+}(SM))$. The general problem in tensor tomography is to determine properties of a function $f$ from its integrals over geodesics as encoded by the transform $If$.

\begin{Question}
Given $f \in C^{\infty}(SM)$, what properties of $f$ may be determined from the knowledge of $If$?
\end{Question}

Clearly a general function $f$ on $SM$ is not determined by its geodesic ray transform alone, since $f$ depends on more variables than $If$. In applications one often encounters the transform $I$ acting on special functions on $SM$ that arise from symmetric tensor fields, and we will now consider this case.


Let $f = f_{i_1 \cdots i_m} dx^{i_1} \otimes \cdots \otimes dx^{i_m}$ be a smooth symmetric $m$-tensor field on $M$. Such a tensor field induces a smooth function $f_{m}(x,v)$ on $SM$ by 
\begin{equation*}
f_{m}(x,v )=f_{i_{1}...i_{m}}\left( x\right) v^{i_{1}}...
v^{i_{m}}. 
\end{equation*}
The operator $I_{m},$ defined by 
\begin{equation*}
I_{m}f=If_{m},
\end{equation*} 
is called the {\it geodesic ray transform} of the symmetric
tensor field $f$. If the manifold $(M,g)$ is non-trapping and the
boundary $\partial M$ is strictly convex, then 
$$
I_{m}:C^\infty(M,S_{m}(M))\rightarrow C(\partial_{+}(SM)),
$$
where $S_{m}(M)$ denotes the bundle of symmetric $m$-tensor fields over $(M,g)$. We will frequently identify the tensor field $f$ on $M$ with the function $f_m$ on $SM$ (see \cite{PSU} for more details).

It is known that any symmetric smooth enough tensor field $f$ may
be decomposed in a potential and solenoidal part \cite{Sh}:
\begin{equation*}\label{sol}
f=f^s+dp,\qquad \delta f^s=0,\ \ p|_{\partial M}=0,
\end{equation*}
where $p$ is a smooth symmetric $(m-1)$-tensor field on $M$, the inner derivative $d=\sigma \nabla$ is the
symmetric part of the covariant derivative $\nabla$, and $\delta$ is the divergence (the adjoint of $-d$ in the natural $L^2$ inner product). If $f$ is a $1$-tensor, identified with a vector field $W$, this generalizes the usual Helmholtz decomposition of a vector field, 
$$
W = W^s + \text{grad}(p), \qquad \mathrm{div}(W^s) = 0, \ \ p|_{\partial M} = 0.
$$
It is easy to see, using the fact that $p$ vanishes on $\partial M$, that the
geodesic ray transform of the potential part $dp$ is zero. We
denote by $C^\infty_{{sol}}(M,S_m(M))$ the space of smooth
solenoidal $m$-symmetric tensor fields. The remark above means that we can only expect to recover the solenoidal part of a tensor field from its ray transform. This leads to the following definition of solenoidal injectivity, or $s$-injectivity for short.

\begin{Definition}
The ray transform on symmetric $m$-tensors, $m \geq 1$, is said to be \emph{$s$-injective} if $I_m f=0$ implies $f^s=0$ for any $f \in C^\infty(M,S_{m}(M))$. In the case of functions on $M$ ($m=0$), $I_0$ is said to be $s$-injective if $I_0 f = 0$ implies $f = 0$ for any $f \in C^{\infty}(M)$.
\end{Definition}

The transforms $I_m$ arise in several applications as well as in the boundary rigidity problem. The latter consists in determining the Riemannian metric of a compact Riemannian manifold with boundary, modulo isometries fixing the boundary, from the distance function $d_g|_{\partial M \times \partial M}$ between boundary points \cite{Mi}.
The case of $I_0$ when the metric
is Euclidean is the standard X-ray transform that integrates a function along lines. Radon found
in 1917 an inversion formula to determine a
function knowing the X-ray transform. 
Inversion formulas of this type have been implemented
numerically using the filtered backprojection algorithm which is
used today in CT scans.

Another important transform in medical
imaging and other applications is the Doppler transform which
integrates a vector field along lines. This corresponds to the case
of $I_1$ for the case of the Euclidean metric. The motivation is
ultrasound Doppler tomography. It is known that blood flow is
irregular and faster around tumor tissue than in normal tissue and
Doppler tomography attempts to reconstruct the blood flow pattern.
Mathematically the problem is to what extent a vector field is
determined from its integral along lines.

The case of integration along more general geodesics arises in geophysical imaging in determining the inner structure of the Earth since the speed of elastic waves generally increases with depth, thus curving the rays back to the Earth surface. It also arises in ultrasound imaging. The geodesic ray transform $I_0$, that is, the integration of a function along geodesics, arises as the linearization of the boundary rigidity problem in a
conformal class of metrics. The linearization of the boundary rigidity problem itself leads to $I_2$, i.e.~the integration of tensors of order two along geodesics. The case of
integration of tensors of order 4 along geodesics arises in certain inverse problems in elasticity \cite{Sh}.

Many of the results in this survey are valid in the case when $(M,g)$ is {\it simple}, a notion that naturally arises in the context of the boundary rigidity problem \cite{Mi}. We recall that a Riemannian manifold with boundary is said to be simple if the boundary is strictly convex and if any two points are connected by a unique geodesic depending smoothly on the endpoints. In particular, a simple manifold is nontrapping and has no conjugate points.

One of the main results we review in this paper is the $s$-injectivity of $I_m$ for all $m$ for simple two-dimensional manifolds that was proven recently in \cite{PSU}.

\begin{Theorem} \label{theorem_sinjectivity} If $(M,g)$ is a simple two-dimensional manifold, then $I_m$ is $s$-injective for any $m \geq 0$.
\end{Theorem}

This result was known earlier for $m=0$ \cite{Mu}, $m=1$ \cite{AR} and $m=2$ \cite{Sh1}. A key point in proving the result for general $m$ is the efficient use of surjectivity properties of $I_0^*$, the adjoint of $I_0$. In fact, \cite{PSU} gave the following more general result.

\begin{Theorem} \label{theorem_sinjectivity_second} If $(M,g)$ is a compact non-trapping two dimensional manifold with strictly convex boundary, and if $I_0$ and $I_1$ are $s$-injective and $I_0^*$ is surjective, then $I_m$ is $s$-injective for any $m \geq 0$.
\end{Theorem}

To describe in detail the adjoint $I_0^*$, for any function $w$ on $\partial_+(SM)$ we define the function 
$$
w_{\psi}(x,v) = w(\varphi_{-\tau(x,-v)}(x,v)), \quad (x,v) \in SM.
$$
Then the solution of the boundary value problem for the transport
equation
$$Xu=0 \text{ in }SM,\quad u|_{\partial _{+}(SM)}=w$$
is equal to $u=w_{\psi}.$

Recall that $I$ is a bounded map $L^{2}(SM)\rightarrow L^{2}_{\mu}(\partial_{+}(SM))$. The adjoint $I^{\ast}$ is bounded $L^{2}_{\mu}(\partial_{+}(SM))\rightarrow L^{2}(SM)$, and it is easy to compute explicitly. In the case of $I_0$, for $f\in C^{\infty}(M)$ and $w\in C^{\infty}(\partial_{+}(SM))$, we have 
$$
\begin{array}{rcl}(I_0f,w)_{L^{2}_{\mu}(\partial_{+}(SM))} &=& \displaystyle\int_{\partial_{+}(SM)}\displaystyle\int^{\tau(x,v)}_{0}f(\varphi_{t}(x,v))w_{\psi}(\varphi_{t}(x,v))\mu\,dtd\Sigma^{2n-2} \\
&=& \displaystyle\int_{SM}fw_{\psi}d\Sigma^{2n-1} \\&=& \displaystyle\int_{M}f(x)\left(\displaystyle\int_{S_{x}}w_{\psi}(x, v)\,dS_{x}(v)\right)\,dV^{n}(x).\\
\end{array}$$
The second equality used Santal\'o's formula \cite{DPSU}. From this computation we conclude that
$$I_0^{\ast}w(x)=\displaystyle\int_{S_{x}}w_{\psi}(x,v)\,dS_{x}(v).$$
Similarly, the adjoint of $I_{m}$ is the operator $I_{m}^{\ast }:L_{\mu }^{2}\left(\partial _{+}(SM) \right) \rightarrow L^{2}\left( M,S_{m}(M\right))$ which is given by
\[ \left( I_{m}^{\ast
}w\right) ^{_{i_{1}...i_{m}}}\left( x\right) =\int\limits_{S_{x}}
w_{\psi }(x,v)v^{i_{1}}...v^{i_{m}} \,dS _{x}(v). \]


\begin{Definition}
We say that $I_0^*$ is surjective if for any $f \in C^{\infty}(M)$, there is a function $w \in C^{\infty}(\partial_+(SM))$ with $I_0^* w = f$ in $M$ and $w_{\psi} \in C^{\infty}(SM)$.
\end{Definition}

The surjectivity of $I_0^*$ in the above sense was proved in \cite{PU} on simple manifolds of any dimension. We will show below how this result is used in the uniqueness proof of tensor tomography in two dimensions. 

In this paper, we also review results in higher dimensions. Here is a summary of what is known about $s$-injectivity on simple manifolds of dimension $n \geq 2$:

\begin{itemize}

\item $I_0$ is injective \cite{Mu}.

\item $I_1$ is $s$-injective \cite{AR}.

\item $I_m$ is $s$-injective for all $m$ if $n=2$ \cite{PSU}.

\item $I_m$ is $s$-injective for all $m$ for manifolds of negative sectional curvature \cite{PS}, or under certain other curvature restrictions \cite{D}, \cite{Pe}, \cite{Sh}.

\item $I_2$ is $s$-injective for generic simple metrics including real-analytic ones \cite{SU3}.

\end{itemize}
See \cite{D}, \cite{Sh1/2}, \cite{Sh3/2}, \cite{SU4}, \cite{UV} for uniqueness results on certain non-simple manifolds. We will also review results on the stability and range for  $I_m$, and moreover we propose several open problems.



A brief summary of the contents of this paper is as follows. Section \ref{sec:prelim} contains preliminaries and notation used in the paper. In Section \ref{sec:firstproof} we review the two proofs of Theorem \ref{theorem_sinjectivity} given in \cite{PSU}. In Section \ref{sec:pestov} we explain a natural approach to the proof of the so-called Pestov identities used in
Section \ref{sec:firstproof}. This energy estimate approach resembles Carleman estimates. In Section \ref{sec:microlocal}
we review a microlocal approach to the study of the geodesic ray transform that gives in particular stability estimates which are summarized in Section \ref{sec:stability}. In Section \ref{sec:scatteringrelation} we consider the scattering relation which is used in the characterization of the range and is of independent interest. In Section \ref{sec:range} we state the result of \cite{PSU3} on the range of the geodesic ray transform. In Section \ref{sec:connections} we summarize several results for the attenuated ray transform for unitary connections proved in \cite{PSU2}. In Section \ref{sec:anosov} we survey the result of \cite{PSU4} on $s$-injectivity of the ray transform on $2$-tensors on closed Anosov surfaces. Finally in Section \ref{sec:openproblems} we state several open problems.

\bigskip

\noindent {\bf Acknowledgements.} \ 
M.S. was supported in part by the Academy of Finland and an ERC starting grant, and G.U. was partly supported by NSF and a Walker Family Endowed Professorship.

\section{Facts about the unit circle bundle} \label{sec:prelim}

This section contains some facts needed for explaining the uniqueness proof for tensor tomography on surfaces, and we will restrict our attention to two dimensional manifolds. Let $(M,g)$ be a compact oriented two dimensional Riemannian manifold with smooth boundary
$\partial M$. As usual $SM$ will denote the unit circle bundle which is a compact 3-manifold with boundary given by $\partial(SM)=\{(x,v)\in SM:\;x\in \partial M\}$.

Let $X$ denote the vector field associated with the geodesic flow $\varphi_{t}$.
Since $M$ is assumed oriented there is a circle action on the fibers of $SM$ with infinitesimal generator $V$ called the {\it vertical vector field}. It is possible to complete the pair $X,V$ to a global frame
of $T(SM)$ by considering the vector field $X_{\perp}$ defined as the commutator $X_{\perp}:=[X,V]$. There are two additional structure equations given by $X=[V,X_{\perp}]$ and $[X,X_{\perp}]=-KV$
where $K$ is the Gaussian curvature of the surface. Using this frame we can define a Riemannian metric on $SM$ by declaring $\{X,X_{\perp},V\}$ to be an orthonormal basis. This metric coincides with the Sasaki metric on $SM$, and the volume form of this metric will be denoted by $d\Sigma^3$. The fact that $\{ X, X_{\perp}, V \}$ are orthonormal together with the commutator formulas implies that the Lie derivative of $d\Sigma^3$ along the three vector fields vanishes, in other words, the three vector fields preserve the volume form $d\Sigma^3$. See \cite{SiTh} for more details on these facts.

It will be useful to have explicit forms of the three vector fields in local coordinates. Since $(M,g)$ is two dimensional, we can always choose isothermal coordinates $(x_{1},x_{2})$ so that the metric
can be written as $ds^2=e^{2\lambda}(dx_{1}^2+dx_{2}^2)$ where $\lambda$ is a smooth
real-valued function of $x=(x_{1},x_{2})$. This gives coordinates $(x_{1},x_{2},\theta)$ on $SM$ where
$\theta$ is the angle between a unit vector $v$ and $\partial/\partial x_{1}$.
In these coordinates the vertical vector field is just 
$$
V=\frac{\partial}{\partial\theta},
$$
and the other vector fields are given by 
\[X=e^{-\lambda}\left(\cos\theta\frac{\partial}{\partial x_{1}}+
\sin\theta\frac{\partial}{\partial x_{2}}+
\left(-\frac{\partial \lambda}{\partial x_{1}}\sin\theta+\frac{\partial\lambda}{\partial x_{2}}\cos\theta\right)\frac{\partial}{\partial \theta}\right),\]
\[X_{\perp}=-e^{-\lambda}\left(-\sin\theta\frac{\partial}{\partial x_{1}}+
\cos\theta\frac{\partial}{\partial x_{2}}-
\left(\frac{\partial \lambda}{\partial x_{1}}\cos\theta+\frac{\partial \lambda}{\partial x_{2}}\sin\theta\right)\frac{\partial}{\partial \theta}\right).\]

Given functions $u,v:SM\to \C$ we consider the
$L^2$ inner product and norm 
\[(u,v) =\int_{SM}u\bar{v}\,d\Sigma^3, \qquad \norm{u} = (u,u)^{1/2}.\]
Since $X, X_{\perp}, V$ are volume preserving we have $(Vu,v) = -(u,Vv)$ for $u, v \in C^{\infty}(SM)$, and if additionally $u|_{\partial(SM)} = 0$ or $v|_{\partial(SM)} = 0$ then also $(Xu,v) = -(u,Xv)$ and $(X_{\perp} u, v) = -(u, X_{\perp} v)$.

The space $L^{2}(SM)$ decomposes orthogonally
as a direct sum
\[L^{2}(SM)=\bigoplus_{k\in\mathbb Z}H_{k}\]
where $H_k$ is the eigenspace of $-iV$ corresponding to the eigenvalue $k$.
A function $u\in L^{2}(SM)$ has a Fourier series expansion
\[u=\sum_{k=-\infty}^{\infty}u_{k},\]
where $u_{k}\in H_k$. Also $\|u\|^2=\sum\|u_k\|^2$, where $\|u\|^2=(u,u)^{1/2}$. The even and odd parts of $u$ with respect to velocity are given by 
$$
u_+ := \sum_{k \text{ even}} u_k, \qquad u_- := \sum_{k \text{ odd}} u_k.
$$
In the $(x,\theta)$-coordinates previously introduced we may write
\[u_{k}(x,\theta)=\left( \frac{1}{2\pi} \int_0^{2\pi} u(x,t) e^{-ikt} \,dt \right) e^{ik\theta}=\tilde{u}_{k}(x)e^{ik\theta}.\]
Observe that for $k\geq 0$, $u_k$ may be identified with a section of the $k$-th tensor
power of the canonical line bundle; the identification takes $u_k$ into $\tilde{u}_{k}e^{k\lambda}(dz)^k$ where
$z=x_{1}+ix_{2}$.

The next definition introduces holomorphic and antiholomorphic functions with respect to the $\theta$ variable.

\begin{Definition} A function $u:SM\to\C$ is said to be holomorphic
if $u_{k}=0$ for all $k<0$. Similarly, $u$ is said to be antiholomorphic if $u_{k}=0$ for all $k>0$.
\end{Definition}

Let $\Omega_{k}:=H_{k}\cap C^{\infty}(SM)$. As in \cite{GK} we introduce the following
first order elliptic operators 
$$\eta_{+},\eta_{-}:C^{\infty}(SM,\C^n)\to
C^{\infty}(SM,\C^n)$$ given by
\[\eta_{+}:=(X+iX_{\perp})/2,\;\;\;\;\;\;\eta_{-}:=(X-iX_{\perp})/2.\]
Clearly $X=\eta_{+}+\eta_{-}$. 
From the structure equations for the frame $\{X,X_{\perp},V\}$ one easily derives:
\[\eta_{+}:\Omega_{k}\to \Omega_{k+1},\;\;\;\;\eta_{-}:\Omega_{k}\to \Omega_{k-1},\;\;\;\;(\eta_{+})^{*}=-\eta_{-}.\]

We will also employ the fiberwise Hilbert transform $H: C^{\infty}(SM) \to C^{\infty}(SM)$, defined in terms of Fourier coefficients as 
$$
Hu_k := -i \,\text{sgn}(k) u_k.
$$
Here $\text{sgn}(k)$ is the sign of $k$, with the convention $\text{sgn}(0) = 0$. Thus, $u$ is holomorphic iff $(\id - iH)u = u_0$ and antiholomorphic iff $(\id + iH)u = u_0$.

The following commutator formula for the Hilbert transform and the geodesic vector field, proved in \cite{PU}, has been a crucial component for many results reviewed in this paper. 

\begin{Proposition} \label{prop:hxcommutator}
Let $(M,g)$ be a two dimensional Riemannian manifold. For any
smooth function $u$ on $SM$ we have the identity
\begin{equation*}
[H, X]u=X_{\perp}u_{0}+(X_{\perp}u)_{0}
\end{equation*}
where
\[
u_{0}(x)=\frac{1}{2\pi}\int_{S_{x}}u(x,v) \,dS_{x}
\]
is the average value.
\end{Proposition}

\begin{proof} It suffices to show that
\[[\id+iH,X] u = iX_{\perp} u_0 + i(X_{\perp} u)_0.\]
Since $X=\eta_{+}+\eta_{-}$ we need to compute $[\id+iH,\eta_{\pm}]$, so let us find
$[\id+iH,\eta_{+}]u$, where $u=\sum_k u_k$. Recall that $(\id+iH)u=u_0+2\sum_{k\geq 1}u_k$.
We find:
\[(\id+iH)\eta_{+}u=\eta_{+}u_{-1}+2\sum_{k\geq 0}\eta_{+}u_{k},\]
\[\eta_{+}(\id+iH)u=\eta_{+}u_0+2\sum_{k\geq 1}\eta_{+}u_k.\]
Thus
\[[\id+iH,\eta_{+}]u=\eta_{+}u_{-1}+\eta_{+}u_0.\]
Similarly we find
\[[\id+iH,\eta_{-}]u=-\eta_{-}u_0-\eta_{-}u_{1}.\]
Therefore using that $iX_{\perp}=\eta_{+}-\eta_{-}$ we obtain
\[[\id+iH,X]u=iX_{\perp}u_0+i(X_{\perp}u)_{0}\]
as desired.
\end{proof}

\section{Tensor tomography on surfaces}
\label{sec:firstproof}

The paper \cite{PSU} gave two proofs for uniqueness in tensor tomography on a simple surface $(M,g)$. In this section we will give an outline of both proofs. They are based on Pestov identities, which are energy estimates for operators related to the ray transform, and which will be discussed in more detail in Section \ref{sec:pestov}. Below we will make use of the concepts introduced in Sections \ref{sec:intro} and \ref{sec:prelim}.

\bigskip

\noindent {\bf First proof.} \ 
To explain the idea behind the first proof of $s$-injectivity, let us first assume that $f$ is a $0$-tensor, that is, $f \in C^{\infty}(M)$. Assuming that $I_0 f = 0$, it is required to show that $f = 0$. The first step is a reduction from the integral operator $I_0$ into a PDE question involving a transport equation. The function 
$$
u(x,v) = \int_0^{\tau(x,v)} f(\varphi_t(x,v)) \,dt, \quad (x,v) \in SM
$$
solves the transport equation 
$$
Xu = -f \ \ \text{in } SM, \quad u|_{\partial(SM)} = 0.
$$
 It is enough to show that $u = 0$, since then also $f=0$.

Isothermal coordinates allow to identify 
$$
SM = \{ (x,\theta) \,;\, x \in \overline{\mathbb{D}}, \,\theta \in [0,2\pi) \}.
$$
The vertical vector field on $SM$ is $V = \frac{\partial}{\partial \theta}$. 
We want to show that 
$$
\left\{ \begin{array}{c} Xu = -f \\ u|_{\partial(SM)} = 0 \end{array} \right. \implies u = 0.
$$
If $f$ is a $0$-tensor, $f = f(x)$, then $Vf = 0$. Thus it is enough to show that 
$$
\left\{ \begin{array}{c} VXu = 0 \\ u|_{\partial(SM)} = 0 \end{array} \right. \implies u = 0.
$$

This calls for a uniqueness result for the operator $P = VX$. In isothermal coordinates, this operator has the form 
$$
P = e^{-\lambda} \frac{\partial}{\partial \theta} \left( \cos \theta \frac{\partial}{\partial x_1} + \sin \theta \frac{\partial}{\partial x_2} + h(x,\theta) \frac{\partial}{\partial \theta} \right)
$$
where $h(x,\theta)$ is a certain smooth function. It turns out that the operator $P$ is rather exotic and there do not seem to be general results on uniqueness properties of such operators in the literature. Here are some facts about the operator $P$:
\begin{itemize}
\item
it is a second order operator on 3D manifold $SM$
\item
it has multiple characteristics
\item 
$P+W$ has compactly supported solutions for some first order perturbation $W$
\item 
it enjoys a subelliptic type estimate $\norm{u}_{H^1(SM)} \leq C \norm{Pu}_{L^2(SM)}$ for $u \in C^{\infty}(SM)$ with $u|_{\partial(SM)} = 0$.
\end{itemize}

However, we can still prove a global uniqueness result for $P$ by using energy estimates. This involves the Pestov identity in $L^2(SM)$ inner product when $u|_{\partial(SM)} = 0$:
$$
\norm{Pu}^2 = \norm{Au}^2 + \norm{Bu}^2 + (i[A,B]u, u)
$$
where $P = A+iB$, $A^* = A$, $B^* = B$.

We will compute the commutator below, and this gives (see Proposition \ref{prop_pestov_standard})
$$
\norm{Pu}^2 = \norm{XVu}^2 - (KVu,Vu) + \norm{Xu}^2.
$$
It is known \cite{PSU2} that on simple manifolds
$$
\norm{XVu}^2 - (KVu,Vu) \geq 0, \quad u \in C^{\infty}(SM), u|_{\partial(SM)} = 0.
$$
(Note that in the case of non-positive curvature, i.e.~$K \leq 0$, one always has $\norm{XVu}^2 - (KVu,Vu) \geq 0$.) Thus $Pu = 0$ implies $u = 0$, showing injectivity of $I_0$.

We now return to tensor tomography. Let $Xu = -f$ in $SM$, $u|_{\partial(SM)} = 0$ where $f$ is the function on $SM$ corresponding to a symmetric $m$-tensor field. It will be convenient to switch to a slightly different setup and think of $u$ and $f$ (which are functions $SM \to \mC$) as sections of the trivial bundle $E = SM \times \mC$. The transport equation then becomes an equation for sections of $E$, 
$$
D_X^0 u = -f
$$
where $D_X^0 = d$ is the flat connection on the trivial bundle $E$.

One benefit of this (trivial) change of point of view is that from the equation on sections, one sees that the transport equation has a natural gauge group acting via multiplication by smooth functions $c \in C^{\infty}(M)$. This action preserves $m$-tensors, and leads to gauge equivalent equations 
$$
D^A_X(cu) = -cf
$$
where $D^A = d + A$ is a gauge equivalent connection on $E$ and $A = -c^{-1} dc$ is the $1$-form determining the connection.

Now we try to use an energy identity for the connections $D^A$. This \emph{Pestov identity with a connection} is proved in the same way as the usual Pestov identity (see Proposition \ref{prop_pestov_attenuation}), and reads in $L^2(SM)$ norms 
$$
\norm{V(X+A)u}^2 = \norm{(X+A)Vu}^2 - (KVu,Vu) + \norm{(X+A)u}^2 + (*F_A Vu,u).
$$
Here $*$ is Hodge star and 
$$
F_A = dA + A \wedge A
$$
is the curvature of the connection $D^A = d + A$. We observe that if the curvature $*F_A$ and the expression $(Vu,u)$ have suitable signs, we gain a positive term in the energy estimate.

This observation does not immediately lead to anything new since curvature is preserved under gauge transformation. Thus, if $D^A$ is gauge equivalent to $D^0$, then $F_A = F_0 = 0$. However, we can use a \emph{generalized gauge transformation} that arranges a sign for $F_A$. This involves gauge transformations via functions $c$ that may depend on the $v$ variable. Such transformations break the $m$-tensor structure of the equation, but turn out to be manageable if the gauge transforms are \emph{holomorphic} in a suitable sense.

Recall from Section \ref{sec:prelim} that a function $u \in L^2(SM)$ is called \emph{holomorphic} if $u_k = 0$ for $k < 0$. The main point is the following theorem guaranteeing that holomorphic gauge transformations always exist. This is related to injectivity of the attenuated ray transform on simple surfaces \cite{SaU}, and in the form below it is proved in \cite{PSU} and \cite{PSU2}. The proof is based on the surjectivity of $I_0^*$.

\begin{Theorem}[Holomorphic gauge transformation] \label{thm_holomorphic_integrating_factors}
If $A$ is a $1$-form on a simple surface, there is a holomorphic $w \in C^{\infty}(SM)$ such that $X + A = e^w \circ X \circ e^{-w}$.
\end{Theorem}
\begin{proof}
Since $M$ is simply connected, there is a Hodge decomposition $A_j \,dx^j = da + \star db$ for some $a, b \in C^{\infty}(M)$ ($\star$ is the Hodge star operator). In terms of the corresponding functions on $SM$ we have $A = Xa + X_{\perp} b$. Replacing $w$ by $w-a$, it is enough to consider the case where $A = X_{\perp} b$.

Let us try a solution of the form $w = (\id + iH) \hat{w}$ where $\hat{w} \in C^{\infty}(SM)$ is even with respect to $v$. By Proposition \ref{prop:hxcommutator}, 
$$
Xw = (\id + iH) X \hat{w} - i[H,X] \hat{w} = (\id + iH) X \hat{w} - i X_{\perp} \hat{w}_0.
$$
Now it is sufficient to find $\hat{w}$ even with $X \hat{w} = 0$ and $\hat{w}_0 = -ib$. Using the surjectivity of $I_0^*$ \cite{PU}, there is some $h \in C^{\infty}(\partial_+(SM))$ with $I_0^* h = -2\pi ib$. But if $w' \in C^{\infty}(SM)$ is the function with $X w' = 0$ in $SM$ and $w'|_{\partial_+(SM)} = h$, we have $(w')_0 = \frac{1}{2\pi} I_0^* h = -ib$. It is enough to take $\hat{w}$ to be the even part of $w'$ with respect to $v$.
\end{proof}

We can now explain the end of the proof of the uniqueness result for tensor tomography on simple surfaces. Let $f = \sum_{k=-m}^m f_k$ be an $m$-tensor written in terms of its Fourier components, and let 
$$
Xu = -f, \quad u|_{\partial(SM)} = 0.
$$
Choose a primitive $\varphi$ of the volume form $\omega_g$ of $(M,g)$, so that $d\varphi = \omega_g$. Let $s > 0$ be large, let $A_s = -is \varphi$, and choose a holomorphic $w$ with $X + A_s = e^{sw} \circ X \circ e^{-sw}$. The transport equation becomes 
$$
(X+A_s)(e^{sw} u) = -e^{sw} f, \quad e^{sw} u|_{\partial(SM)} = 0.
$$
Here the curvature of $A_s$ has a sign and one has information on Fourier coefficients of $e^{sw} f$. The Pestov identity with connection allows to control Fourier coefficients of $e^{sw} u$, eventually proving $s$-injectivity of $I_m$.

Heuristically, the proof above involves ''twisting'' the trivial bundle $E$ by a holomorphic gauge transformation to make it positively curved, using the Pestov identity with a large positive term coming from the connection to absorb error terms, and then undoing the gauge transformation (this is possible because of holomorphicity) to get uniqueness.  This idea of twisting to impose positivity to prove a vanishing theorem is of course
well known in Complex Geometry and it is the way one proves results like the Kodaira vanishing theorem \cite{GH}. Our setting is more complicated since the relevant PDE is the transport equation which is harder to handle than the Cauchy-Riemann equation. However this analogy is important and permeates all work; in particular the injectivity results on the attenuated ray transform for unitary connections, to be discussed later on, are also proved in this fashion.

There is an interesting connection between the Pestov identity with connection $A_s$ above and with Carleman estimates. In fact, the Pestov identity with $A_s$ implies the estimate 
$$
s^{1/2} \norm{u}_{L^2_x \dot{H}^{1/2}_{\theta}} \lesssim \norm{e^{sw} X (e^{-sw} u)}_{L^2_x \dot{H}^1_{\theta}}.
$$
Here we use the norms 
$$
\norm{u}_{L^2_x \dot{H}^s_{\theta}} = \left( \sum_{k \neq 0} \abs{k}^{2s} \norm{u_k}_{L^2(SM)}^2 \right)^{1/2}.
$$
Formally this looks very much like a Carleman estimate with exponential weights, but it involves some slightly exotic spaces and one can see that the positivity comes from $\mathrm{Im}(w)$ (not $\mathrm{Re}(w)$ as is usual in Carleman estimates)! We finally remark that such an estimate is sufficient for 
\begin{itemize}
\item 
absorbing large attenuation (even for systems, see Section \ref{sec:connections})
\item 
absorbing error terms coming from $m$-tensors.
\end{itemize}
However, it seems that the estimate may not be enough to 
\begin{itemize}
\item 
localize in space
\item 
absorb error terms coming from curvature of $M$.
\end{itemize}

\bigskip

\noindent {\bf Second proof.} \ 
Next we explain a very short alternative proof to a key step in the injectivity result.

Suppose that $u$   is a smooth solution of $Xu = -f$   in $SM$ where $f_k = 0$ for $k \leq -m-1$ and $u|_{\partial(SM)} = 0$. We wish to show that $u_k = 0$ for $k \leq -m$. This, together with the analogous result for positive Fourier coefficients, implies that $f = Xh$ where the Fourier expansion of $h$ has degree $m-1$ and $h|_{\partial(SM)} = 0$, thus proving $s$-injectivity.

We choose a nonvanishing function $h \in \Omega_m$. In fact, in isothermal coordinates, we can set 
$$
h(x,y,\theta) := e^{im\theta}.
$$
Define the $1$-form 
$$
A := -h^{-1} Xh.
$$
Then $hu$ solves the problem 
$$
(X + A)(hu) = -hf \text{ in } SM, \quad hu|_{\partial(SM)} = 0.
$$
Note that $hf$ is a holomorphic function. Next we employ a holomorphic integrating factor, as above: by Theorem  \ref{thm_holomorphic_integrating_factors} there exists a holomorphic $w \in C^{\infty}(SM)$ with $Xw = A$. The function $e^w h u$ then satisfies 
$$
X(e^w h u) = -e^w h f \text{ in } SM, \quad e^w h u|_{\partial(SM)} = 0.
$$
The right hand side $e^w h f$ is holomorphic. It is known that the solution $e^w h u$, which vanishes on $\partial(SM)$� also has to be holomorphic and further $(e^w h u)_0 = 0$. This follows from the $s$-injectivity of $I_0$ and $I_1$ (see \cite{PSU}, \cite{SaU}). Looking at Fourier coefficients shows that $(h u)_k = 0$   for $k \leq 0$, and therefore $u_k = 0$ for $k \leq -m$ as required.

\section{Pestov identity} 
\label{sec:pestov}

In this section we consider the Pestov identity, which is the basic energy identity that has been used since the work of Mukhometov \cite{Mu} in most injectivity proofs of ray transforms in the absence of real-analyticity or special symmetries. Pestov type identities were also used in \cite{AR} to prove s-injectivity for $I_1$ on simple manifolds and in
\cite{PS} to prove s-injectivity for any $m$ in any dimensions if the sectional curvatures are negative. See \cite{D}, \cite{Pe}, \cite{Sh} for further results. Pestov identities have often appeared in a somewhat ad hoc way, but here we follow \cite{PSU} which gives a new point of view making the derivation of these identities more transparent. We will only consider two dimensional manifolds in this section.

The easiest way to motivate the Pestov identity is to consider the injectivity of the ray transform on functions. The first step, as discussed in Section \ref{sec:first proof}, is to recast the injectivity problem as a uniqueness question for the partial differential operator $P$ on $SM$ where 
$$
P := VX.
$$
This involves a standard reduction to the transport equation.

\begin{Proposition}
Let $(M,g)$ be a compact oriented nontrapping surface with strictly convex smooth boundary. The following statements are equivalent.
\begin{enumerate}
\item[(a)]
The ray transform $I: C^{\infty}(M) \to C(\partial_+(SM))$ is injective.
\item[(b)]
Any smooth solution of $Pu = 0$ in $SM$ with $u|_{\partial(SM)} = 0$  is identically zero.
\end{enumerate}
\end{Proposition}
\begin{proof}
Assume that the ray transform is injective, and let $u \in C^{\infty}(SM)$   solve $Pu = 0$   in $SM$   with $u|_{\partial(SM)} = 0$. This implies that $Xu = -f$   in $SM$ for some smooth $f$   only depending on $x$, and we have $0 = u|_{\partial_+(SM)} = If$. Since $I$   is injective one has $f = 0$   and thus $Xu = 0$, which implies $u = 0$   by the boundary condition.

Conversely, assume that the only smooth solution of $Pu = 0$   in $SM$ which vanishes on $\partial(SM)$ is zero. Let  $f \in C^{\infty}(M)$ be a function with $If = 0$, and define the function 
$$
u(x,v) := \int_0^{\tau(x,v)} f(\gamma(t,x,v)) \,dt, \quad (x,v) \in SM.
$$
This function satisfies the transport equation $Xu = -f$ in $SM$ and $u|_{\partial(SM)} = 0$ since $If = 0$, and also $u \in C^{\infty}(SM)$ (see \cite{PSU2}). Since $f$ only depends on $x$ we have $Vf = 0$, and consequently $Pu = 0$ in $SM$ and $u|_{\partial(SM)} = 0$. It follows that $u = 0$ and also $f = -Xu = 0$.
\end{proof}

We now focus on proving a uniqueness statement for solutions of $Pu = 0$   in $SM$. For this it is convenient to express $P$ in terms of its self-adjoint and skew-adjoint parts in the $L^2(SM)$ inner product as 
$$
P = A + iB, \quad A := \frac{P+P^*}{2},  \   \   B := \frac{P-P^*}{2i}.
$$
Here the formal adjoint $P^*$    of $P$   is given by 
$$
P^* := XV.
$$
In fact, if $u \in C^{\infty}(SM)$   with $u|_{\partial(SM)} = 0$, then 
\begin{align}
\norm{Pu}^2 &= ((A+iB)u, (A+iB)u) = \norm{Au}^2 + \norm{Bu}^2 + i(Bu,Au) - i(Au,Bu)  \label{p_ab_computation} \\
 &= \norm{Au}^2 + \norm{Bu}^2 + (i[A,B]u, u). \notag
\end{align}
This computation suggests to study the commutator $i[A,B]$. We note that the argument just presented is typical in the proof of $L^2$ Carleman estimates \cite{H}.

By the definition of $A$   and $B$ it easily follows that $i[A,B] = \frac{1}{2} [P^*, P]$. By the commutation formulas for $X$, $X_{\perp}$ and $V$, this commutator may be expressed as 
\begin{align*}
[P^*, P] &= XVVX - VXXV = VXVX + X_{\perp} VX - VXVX - VX X_{\perp} \\
 &= V[X_{\perp}, X] - X^2 = -X^2 + VKV.
\end{align*}
Consequently 
$$
([P^*, P]u, u) = \norm{Xu}^2 - (KVu,Vu).
$$
If the curvature $K$ is nonpositive, then $[P^*,P]$ is positive semidefinite. More generally, one can try to use the other positive terms in \eqref{p_ab_computation}. Note that 
$$
 \norm{Au}^2 + \norm{Bu}^2 = \frac{1}{2}(\norm{Pu}^2 + \norm{P^* u}^2).
$$
The identity \eqref{p_ab_computation} may then be expressed as 
\begin{align*}
\norm{Pu}^2 &= \norm{P^* u}^2 + ([P^*,P]u, u).
\end{align*}
(Note that we could have just started from the last identity, but expressing matters via $A$ and $B$ highlights the similarity to Carleman estimates.) Moving the term $\norm{Pu}^2$ to the other side, we have proved the version of the Pestov identity which is most suited for our purposes. The main point in this proof was that the Pestov identity boils down to a standard $L^2$ estimate based on separating the self-adjoint and skew-adjoint parts of $P$ and on computing one commutator, $[P^*, P]$.

\begin{Proposition} \label{prop_pestov_standard}
If $(M,g)$ is a compact oriented surface with smooth boundary, then 
$$
\norm{XVu}^2 - (KVu,Vu) + \norm{Xu}^2 - \norm{VXu}^2 = 0
$$
for any $u \in C^{\infty}(SM)$ with $u|_{\partial(SM)} = 0$.
\end{Proposition}

It is known \cite{DKSU}, \cite{PSU2} that on a simple surface, one has 
$$
\norm{XVu}^2 - (KVu,Vu) \geq 0, \quad u \in C^{\infty}(SM), \ u|_{\partial(SM)} = 0.
$$
Also, if $Xu = -f$ where $f = f_0 + f_1 + f_{-1}$ is the sum of a $0$-form and $1$-form, we have 
$$
\norm{Xu}^2 - \norm{VXu}^2 = \norm{f_0}^2 \geq 0.
$$
These two facts together with the Pestov identity give the standard proof of $s$-injectivity of the ray transform for $0$-forms and $1$-forms on simple surfaces. It is easy to see where this proof breaks down if $m \geq 2$: the Fourier expansion $f = \sum_{k=-m}^m f_k$ implies  
$$
\norm{Xu}^2 - \norm{VXu}^2 = \norm{f_0}^2 - \sum_{2 \leq \abs{k} \leq m} (k^2-1) \norm{f_k}^2.
$$
This term may be negative, and the Pestov identity may not give useful information unless there is some extra positivity like a curvature bound.

Finally, we consider the Pestov identity in the presence of attenuation given by $A(x,v) = A_j(x) v^j$ where $A_j \,dx^j$ is a purely imaginary $1$-form on $M$. We write $A$   both for the $1$-form and the function on $SM$. The geometric interpretation is that $d+A$ is a unitary connection on the trivial bundle $M \times \mC$, and its curvature is the $2$-form 
$$
F_A := dA + A \wedge A.
$$
Then $\star F_A$ is a function on $M$ where $\star$ is the Hodge star. We consider the operator 
$$
P := V(X+A).
$$
Since $\bar{A} = -A$, the formal adjoint of $P$   in the $L^2(SM)$   inner product is 
$$
P^* = (X+A)V.
$$
The same argument leading to Proposition \ref{prop_pestov_standard}, based on computing the commutator $[P^*, P]$, gives the following Pestov identity proved also in \cite[Lemma 6.1]{PSU2}.

\begin{Proposition} \label{prop_pestov_attenuation}
If $(M,g)$ is a compact oriented surface with smooth boundary and if $A$ is a purely imaginary $1$-form on $M$, then 
$$
\norm{(X+A)Vu}^2 - (KVu,Vu) + \norm{(X+A)u}^2 - \norm{V(X+A)u}^2 + (\star F_A Vu,u) = 0
$$
for any $u \in C^{\infty}(SM)$ with $u|_{\partial(SM)} = 0$.
\end{Proposition}

Using the Fourier expansion of $u$, the last term in the identity is given by 
$$
\sum_{k=-\infty}^{\infty} ik(\star F_A u_k, u_k)
$$
This shows that if $u$   is holomorphic and $i\star F_A > 0$, or if $u$ is antiholomorphic and $i\star F_A < 0$, one gains an additional positive term in the Pestov identity. This is crucial in absorbing negative contributions from the term $\norm{(X+A)u}^2 - \norm{V(X+A)u}^2$ when proving $s$-injectivity on tensor fields.

\section{Microlocal approach}
\label{sec:microlocal}

A different approach that is useful to prove $s$-injectivity of $I_m$ in some cases and gives stability estimates as well as reconstruction formulas in some cases was started in \cite{SU1} and developed further in \cite{SU2,SU3, SU4,SU5}. We describe the method in more detail for $I_0.$
Let $(M,g)$ be a simple manifold embedded in a closed manifold $(N,g)$ and let $U$ be a simple neighborhood of $M$ in $N$.

\begin{Theorem}
$I_0^{\ast}I_0$ is an elliptic pseudodifferential operator on $U$ of order -1. 
\end{Theorem}
\begin{proof}
It is easy to see, that
\begin{equation}
\left( I_0^{\ast }I_0 f\right) \left( x\right) = \int\limits_{S
_{x}}dS_{x}\int\limits_{-\tau \left( x,-v \right) }^{\tau
\left( x, v \right) }f\left( \gamma \left(t, x,v \right)
\right) dt= 2\int\limits_{S _{x}}dS
_{x}\int\limits_{0}^{\tau \left( x, v \right) }f\left( \gamma
\left( t,x,v\right) \right) dt.
\end{equation}

Before we continue we make a remark concerning notation. We have
used up to now the notation $\gamma (t, x, v)$ for a geodesic.
But it is known, that a geodesic depends smoothly on the
point $x$ and vector $\xi t\in T_{x}(M).$ Therefore in  what
follows we will also use sometimes the notation $\gamma (x,v t)$
for a geodesic. Since the manifold  $M$ is simple any small
enough neighborhood $U$ (in $\left( N,g\right) $) is also simple
(an open domain is simple if its closure is simple). \ For any
point $x\in U$\ \ there is an open domain $D_{x}^{U}\subset
T_{x}\left( U\right) $ such that exponential map
$exp_{x}:D_{x}^{U}\rightarrow U,\;exp_{x}\eta =\gamma (x,\eta )$\
is a diffeomorphism  onto $U.$ Let $D_{x},\;x\in M$ \ be the
inverse image of $M$, then $exp_{x}(D_{x})=M$ and
$exp_{x}|_{D_{x}}:D_{x}\rightarrow M$ is a diffeomorphism.

Now we change variables in (2), $y=\gamma (x,v t).$ Then
$t=d_g\left( x,y\right)$ and
\[
(I^{\ast }If)\left( x\right) = \int\limits_{M}K\left( x,y\right)
\,f\left( y\right) dy,
\]
where
\[
K\left( x,y\right) =2\frac{\det \left(
exp_{x}^{-1}\right) ^{\prime }\left( x,y\right) \sqrt{\det g\left(
x\right) }}{d_g^{n-1}\left( x,y\right)}.
\]

Notice, that since
\begin{equation}
\gamma (x,\eta )=x+\eta +O(\left| \eta \right| ^{2}),
\end{equation}
it follows, that the Jacobian of the exponential map is $1$
at $0$, and then $\det (exp_{x}^{-1}\left) ^{\prime }( x,x\right)
=1/\det \left( exp_{x}\right) ^{\prime }\left( x,0\right) =1$.
From (3) we also conclude that
\[
d^{2}\left( x,y\right) =G_{ij}\left( x,y\right) \left( x-y\right) ^{i}\left( x-y\right) ^{j},\;\;G_{ij}\left( x,x\right) =
g_{ij}\left( x\right) ,\;\;G_{ij}\in C^{\infty }\left( M\times M\right).
\]
Therefore the kernel of $I_0^*I_0$ can be written in the form
\[
K\left( x,y\right) =\frac{2\det \left( exp_{x}^{-1}\right) ^{\prime }\left( x,y\right)
\sqrt{\det g\left( x\right) }}{\left( G_{ij}\left( x,y\right)
\left( x-y\right) ^{i}\left( x-y\right) ^{j}\right) ^{\left(
n-1\right) /2}}.
\]
Thus \ the kernel $K$ has at the diagonal $x=y$ \ a singularity
of type $\left| x-y\right| ^{-n+1}.$ \ The kernel
\[
K_{0}^{\,}\left( x,y\right) =\frac{2\sqrt{\det g\left( x\right)
}}{\left( g_{ij}\left( x\right) \left( x-y\right) ^{i}\left(
x-y\right) ^{j}\right) ^{\left( n-1\right) /2}}
\]
has the same  singularity. Clearly, the difference $K-K_{0}$ has
a singularity of type $\left| x-y\right| ^{-n+2}.$  Therefore the
principal symbols of both operators coincide.\ The principal
symbol of the integral operator, corresponding to the kernel
$K_{0}$ coincide with its full symbol and is easily\ calculated.\ As
a result
\[
\sigma \left( I_0^{\ast }I_0\right) (x, \xi)=
2\sqrt{\det g\left( x\right) }\int \frac{e^{-
i(y, \xi)}}{\left( g_{ij}\left( x\right) y^{i}y^{j}\right) ^{\left( n-1\right) /2}}dy=
c_{n}\left| \xi \right| ^{-1}.
\]

\end{proof}



Let $g$ be a simple metric in $M$. Extend $g$  near $M$ and let $M_1$ be a simple manifold with boundary so that $M$ is a compact subset of $M_1$.
We will work with $f$ supported in $M$. We assume that $f$ is extended as $0$ outside $M$. 
Choose a smooth function  $\chi$  supported in $M_1$ such that $\chi=1$ near $M$.  
 
It was shown in \cite{SU2} that the normal operator $N_g= I_m^*I_m$ is a pseudodifferential operator of order -1, for $m=0,1,2$ which is elliptic acting on solenoidal tensor fields. We have,
 
\begin{Theorem} \label{thm_par}
Let $g$ be a simple metric in $M$ and let $\chi $ be as before. Then 
one can construct
a pseudodifferential operator  $a_{ijkl}(x,D)$ of order $1$ so that for any symmetric $2$-tensor $f\in L^2(M)$ we have, 
\begin{equation}  \label{t1}
\chi  a_{ijkl}(x,D)\chi (N_g f)^{kl} = f^s_{M_1} +Kf,
\end{equation} 
where $K: L^2(M) \to H^1(M_1)$ is bounded.  Here $f^s_{M_1}$ denotes the solenoidal part of $f$ on $M_1$.
\end{Theorem}
This result was extended to tensor fields of any order in \cite{SSU}.



When $g$ is a real-analytic simple metric it was shown in \cite{SU3} that $I_2$
is $s$-injective. The proof constructs a parametrix as in the previous result with $K$ analytic regularizing, that is $Kf$ is real-analytic on $M_1$ for $f\in L^2(M)$. The idea of the proof for $I_0$ is that if $I_0 f=0, f\in L^2(M)$ then
$f=-Kf$. Since $Kf$ is real-analytic on $M_1$ and supported on $M$ it must be zero. For the details of the proof for $I_2$ see \cite{SU3}.

\section{Stability estimates}
\label{sec:stability}

It was shown in \cite{SU2}, \cite{SSU} that for a simple manifold $s$-injectivity of $I_m$ implies stability estimates.  This is based on the fact that $N_g = I_m^* I_m$ is an elliptic pseudodifferential operator acting on solenoidal tensor fields. 
We have the following stability estimate for $I_0$ (\cite{SU2}):

\begin{Theorem}  \label{thm_f}
Let $g$ be a simple metric in $M$ and assume that $g$ is extended smoothly as a simple metric near the simple manifold $M_1\supset\supset M$.  Then for any function $f\in L^2(M)$, 

$$
\|f\|_{L^2(M)}/C\le \| N_g f\|_{H^1(M_1)} \le C\|f\|_{L^2(M)}.$$ 
\end{Theorem}

Similarly $s$-injectivity of $I_1$ implies the stability estimate:
\begin{Theorem}  \label{thm_form} Assume that $g$ is simple metric in $M$ and extend $g$ as a simple metric in $M_1\supset\supset M$. Then for any 1-form $f=f_idx^i$ in $L^2(M)$ we have
$$\left\|f^s\right\|_{L^2(M) }/C \le \| N_g  f\|_{  H^1(M_1)} \le C \left\|f^s\right\|_{L^2(M) } .$$
\end{Theorem} 

A sharp stability estimate for $I_2$, assuming that $I_2$ is known to be $s$-injective, was proved in \cite{Stefanov_sharp}:

\begin{Theorem}  \label{thm_g}
Let $g$ be a simple metric in $M$ and assume that $g$ is extended smoothly as a simple metric near the simple manifold $M_1\supset\supset M$. Also assume that $I_2$ is $s$-injective. Then for any symmetric $2$-tensor field $f$ in $L^2(M)$, 

$$
\|f^s\|_{L^2(M)}/C\le \| N_g f\|_{H^1(M_1)} \le C\|f^s\|_{L^2(M)}.$$ 
\end{Theorem}

In order to describe possible stability estimates for $I_m$, we describe an earlier result for $I_2$. In order to state the result we first take boundary normal coordinates
$x^1,..., x^n$ with $x^n=0$ the defining function of $\partial M$.
Introduce the space $\tilde H^1(M_1)$  with norm equal to the $L^2$ norm outside a neighborhood of $\partial M$ and near $\partial M$ (but outside $M$) having the following form in normal local coordinates:
\begin{equation}      \label{norm}|f|_{\tilde H^1(M_1)}^2 = \int_{M_1} \left(\sum_{i=1}^{n-1} |\partial_i f|^2 + |x^n\partial_{n} f|^2 +|f|^2\right)dx, \quad \mbox{ supp } f\subset U.
\end{equation}Here $U$ is a small neighborhood of a point on $\partial M$ and the norm in $\tilde H^1(M_1)$ is defined by using a partition of unity.  

Next we define the norm $$\|N_gf\|_{\tilde H^2(M_1)} = \sum_{i=1}^{n}\|\partial_i N_gf\|_{\tilde H^1(M_1)}   +\|N_g f\|_{  H^1(M_1)}.$$

The earlier stability result for $I_2$ is:
\begin{Theorem}  \label{thm_ten}
Assume that $g$ is simple metric in $M$ and extend $g$ as a simple metric in i
$M_1\supset\supset M$. 

(a) The following estimate holds for each symmetric 2-tensor $f$ in $H^1(M)$:
$$\left\|f^s\right\|_{L^2(M)} \le C  \| N_gf\|_{\tilde H^2(M_1)} +C_s\| f\|_{  H^{-s}(M_1)}, \quad \forall s>0.$$

(b) $\mbox{\rm Ker\,} I_2\cap \mathcal{S}L^2(M)$ is finite dimensional and included in  $C^\infty(M)$. ($\mathcal S$ stands for solenoidal). 

(c) Assume that $I_2$ is s-injective in $M$, i.e., that $\mbox{\rm Ker\,} I_g\cap \mathcal{S}L^2(M)=\{0\}$.
Then for any symmetric 2-tensor $f$ in $H^1(M)$ we have
\begin{equation*}\label{est_N}\left\|f^s\right\|_{L^2(M) } \le C \| N_g  f\|_{\tilde  H^2(M_1)}.
\end{equation*}
\end{Theorem}

This result was proven in \cite{SU2} for the case $m=2$. Using the results of \cite{SSU}, stability estimates of this type can be shown to be valid for any $m$. 

\section{The scattering relation}
\label{sec:scatteringrelation}

To state the results for the range of $I_m$ for simple surfaces we need to recall the definition of the scattering relation which is a subject of interest in its own right.

Suppose we have a Riemannian metric in Euclidean space which is
the Euclidean metric outside a compact set. The inverse scattering
problem for metrics is to determine the Riemannian metric by
measuring the scattering operator (see \cite{G}). A similar
obstruction to the boundary rigidity problem occurs in this case
with the diffeomorphism $\psi$ equal to the identity outside a
compact set. It was proven in \cite{G} that from the wave front set of
the scattering 
operator, one can determine, under some non-trapping assumptions
on the metric, the scattering relation on the boundary of a
large ball. This uses high frequency information of the scattering
operator. In the semiclassical setting Alexandrova has shown for a
large class of operators  that the scattering operator associated
to potential and metric perturbations of the Euclidean Laplacian
is a semiclassical Fourier integral operator quantized by the
scattering relation  \cite{A2}. The scattering relation maps the
point and direction of a geodesic entering the manifold to the
point and direction of exit of the geodesic.

We proceed to define in more detail the scattering relation. To do this, let $\tau^{0}=\tau|_{\partial(SM)}$ and note that this function is equal to zero on $\partial_{-}(SM)$ and is smooth on
$\partial_{+}(SM)$. Its odd part with respect to $v$,
$$\tau _{-}^{0}(x,v)=\frac{1}{2}\left( \tau
^{0}(x,v)-\tau ^{0}\left( x,-v \right) \right),$$ is a smooth
function on $\partial(SM)$ (see for instance \cite{DPSU}).

\begin{Definition}
Let $(M,g)$ be non-trapping with strictly convex boundary. The
scattering relation $\alpha :\partial(SM)
\rightarrow
\partial(SM)$ is defined by
$$\alpha (x,v)=
(\gamma (x, v ,2\tau _{-}^{0}(x,v)),\dot{\gamma}(x,v ,2\tau
_{-}^{0}(x,v ))).
$$
\end{Definition}

The scattering relation is a diffeomorphism $\partial 
(SM) \rightarrow \partial(SM).$
Notice that $\alpha |_{\partial _{+}(SM)
}:\partial _{+}(SM) \rightarrow
\partial _{-}(SM) ,$ $\alpha |_{\partial
_{-} SM }:\partial _{-}(SM)
\rightarrow \partial _{+}(SM) $ are
diffeomorphisms as well. The manifold of inner vectors
$\partial _{+}(SM) $ and outer vectors $\partial
_{-}(SM) $ intersect at the set of tangent
vectors
$$\partial _{0}(SM)
=\{(x, v)\in \partial(SM),\ \langle \nu(x) ,v \rangle=0\;\}.$$
Obviously, $\alpha $ is an involution,
$\alpha ^{2}=id$ and $\partial _{0}(SM) $ is the
hypersurface of its fixed\ points,\ $\alpha (x,v)=(x,v),\;(x,v)\in \partial _{0}(SM) .$

A natural inverse problem is whether the scattering relation
determines the metric $g$ up to an isometry which is the identity
on the boundary. This information takes into account all the
travel times, not just the first arrivals like the boundary distance function. 

We remark that in the case that $(M,g)$ is a simple manifold, and we know the
metric at the boundary (and this is determined if $d_g$ is known), knowing the scattering relation is equivalent to
knowing the boundary distance function (\cite{Mi}).

We introduce the operators of even and odd continuation with
respect to $\alpha $:
\begin{equation*}
A_{\pm }w(x,v)=w(x,v),\;\;\;\;\;\; (x,v)\in \partial
_{+} SM,
\end{equation*}
\begin{equation*}
A_{\pm }w(x,v)=\pm \left( \alpha ^{\ast }w\right) (x,v),
\;\;(x,v)\in \partial _{-}(SM).
\end{equation*}

We will examine next the boundedness properties of $A_{-}, A_{+}$.  

\begin{Lemma}
$A_{\pm}:L^{2}_{\mu}(\partial_{+}(SM))\rightarrow L^{2}_{|\mu|}(\partial(SM))$ are bounded.
\end{Lemma}
\begin{proof}
$$
\begin{array}{rcl}
\|A_{\pm}w\|^{2}_{L^{2}_{|\mu|}(\partial(SM))} &=& \displaystyle\int_{\partial_{+}(SM)}w^{2}\mu\,d\Sigma^{2n-2}+\displaystyle\int_{\partial_{-}(SM)}(\alpha^{\ast}w)^{2}(-\mu\,d\Sigma^{2n-2})\\
 &=& \displaystyle\int_{\partial_{+}(SM)}w^{2}\mu\,d\Sigma^{2n-2}+\displaystyle\int_{\partial_{+}(SM)}w^{2}\alpha^{\ast}(-\mu\,d\Sigma^{2n-2})\\
\end{array}
$$
where $\alpha:\partial_{+}(SM)\rightarrow\partial_{-}(SM)$ is a diffeomorphism. Thus it is enough to show that
$$\alpha^{\ast}(-\mu d\Sigma^{2n-2})=\mu d\Sigma^{2n-2}$$

Let $w\in C^{\infty}(\partial_{+}(SM))$. Then
$$\displaystyle\int_{\partial_{+}(SM)}w\tau\mu\,d\Sigma^{2n-2}=\displaystyle\int_{\partial_{+}(SM)}\int^{\tau(x,v)}_{0}w_{\psi}(\varphi_{t}(x,v))\mu\,dtd\Sigma^{2n-2}=\displaystyle\int_{SM}w_{\psi}\,d\Sigma^{2n-1}$$ 
Set $\tilde{u}(x,v)=u(x,-v)$ for $u\in C^{\infty}(SM)$, one has
$$
\begin{array}{rcl}
\displaystyle\int_{SM}w_{\psi}\,d\Sigma^{2n-1} &=& \displaystyle\int_{SM}\tilde{w}_{\psi}\,d\Sigma^{2n-1} \\
 																											&=& \displaystyle\int_{\partial_{-}(SM)}\displaystyle\int^{\tau(y,-\eta)}_{0}\tilde{w}_{\psi}(\varphi_{t}(y,-\eta))(-\mu)\,dtd\Sigma^{2n-2} \\
 																											&=& \displaystyle\int_{\partial_{-}(SM)}\displaystyle\int^{\tau(y,-\eta)}_{0}w(\alpha(y,\eta))(-\mu)\,dtd\Sigma^{2n-2} \\
 																											&=& \displaystyle\int_{\partial_{+}(SM)}w\tau\alpha^{\ast}(-\mu d\Sigma^{2n-2})\\
\end{array}
$$
Varying $w$ shows that $\alpha^{\ast}(-\mu d\Sigma^{2n-2})=\mu d\Sigma^{2n-2}$ on $\partial_{+}(SM)\backslash\partial_{0}SM$.
\end{proof}

The adjoint $A^{\ast}_{\pm}:L^{2}_{|\mu|}(\partial(SM))\rightarrow L^{2}_{\mu}(\partial_{+}(SM))$ satisfies
$$
\begin{array}{rcl}
(A_{\pm}w,u)_{L^{2}_{|\mu|}(\partial(SM))} &=& \displaystyle\int_{\partial_{+}(SM)}wu\mu\,d\Sigma^{2n-2}\pm\displaystyle\int_{\partial_{-}(SM)}(w\circ\alpha)u(-\mu\,d\Sigma^{2n-2})\\
  &=& \displaystyle\int_{\partial_{+}(SM)}w(u\pm u\circ\alpha)\mu\,d\Sigma^{2n-2} \\
\end{array}
$$
so $A^{\ast}_{\pm}u=(u\pm u\circ\alpha)|_{\partial_{+}(SM)}$. 

In \cite{PU} the following characterization of the space of smooth
solutions of the transport equation was given. Here we define 
$$
C_{\alpha}^{\infty}(\partial_{+}(SM))=\{Êw\in C^{\infty}(\partial_{+}(SM)) : w_{\psi} \in C^{\infty}(SM) \}.
$$

\begin{Lemma}
$$C_{\alpha}^{\infty}(\partial_{+}(SM))=
\{w\in C^{\infty}(\partial_{+}(SM)):A_{+}w\in
C^{\infty}(\partial(SM))\}.$$
\end{Lemma}

Then $I_0^{\ast}w\in C^{\infty}(M)$ whenever $w\in C^{\infty}_{\alpha}(\partial_{+}(SM))$.


We conclude this section by defining certain operators which combine the operators $A_{\pm}$ introduced above with the fibrewise Hilbert transform $H$. These operators will be essential to determine the range of the ray transform in the next section.
Set $H_{\pm}u=Hu_{\pm}$ where $u_+$ (resp. $u_{-}$) denote the even (resp. odd) part of $u\in C^{\infty}(SM)$.

We define
\begin{equation}
P_{-}=A_{-}^{\ast }H_{-}A_{+}, \quad P_{+}=A_{-}^{\ast}H_{+}A_{+}.
\label{eq:p}
\end{equation}

\section{Range of the geodesic ray transform}
\label{sec:range}

We now give the characterization of the range of $I_0$ and $I_1$
in terms of the scattering relation only. We have that these are
the projections of the operators $P_{-}, P_{+}$ respectively (defined in (\ref{eq:p})). For
the details see \cite{PU2}.

\begin{Theorem}
Let $(M,g)$ be simple two dimensional compact Riemannian manifold
with boundary. Then
\begin{enumerate}
\item  A function $u\in C^{\infty }\left(
\partial _{+}(SM) \right)$ belongs to the range of $I_{0}$
iff $u=P_{-}w$ where $w\in C_{\alpha }^{\infty }\left(
\partial _{+}(SM) \right).$
\item  A function $u\in C^{\infty }\left(
\partial _{+}(SM) \right)$ belongs to the range of $ I_{1}$
iff $u=P_{+}w$ where $w\in C_{\alpha }^{\infty }\left(
\partial _{+}(SM) \right).$
\end{enumerate}
\label{thm:range}
\end{Theorem}

We now move on to describe the range of the geodesic ray transform for tensors of order $\geq 2$.
For this we apply the ideas of the second proof of Theorem \ref{theorem_sinjectivity} described in
Section 3. For the details see \cite{PSU3}.

Let $(M,g)$ be a simple surface. The metric $g$ induces a complex structure on $M$ and let $\kappa$ be the canonical line bundle (which we may identify with $T^*M$).
Recall that $H_m$ ($m\in\Z$) is the set of functions in $f\in L^{2}(SM,\C)$ such that
$Vf=imf$. 
The set $\Omega_m=H_{m}\cap C^{\infty}(SM,\C)$ can be identified with the set $\Gamma(M,\kappa^{\otimes m})$ of smooth sections of $m$-th tensor power of the canonical line bundle $\kappa$. This identification depends on the metric and is explained in detail in \cite[Section 2]{PSU4}, but let us give a brief description of it.  Given a section $\xi\in \Gamma(M,\kappa^{\otimes m})$ we can obtain a function on $\Omega_m$ simply by restriction to $SM$: $\xi$ determines the function $SM\ni (x,v)\mapsto \xi_{x}(v^{\otimes m})$ and this gives a 1-1 correspondence. 

Since $M$ is a disk, there is $\xi\in  \Gamma(M,\kappa)$ which is nowhere vanishing. Having picked this section we may define a function 
$h:SM\to S^{1}$ by setting $h(x,v)=\xi_{x}(v)/|\xi_{x}(v)|$.  By construction $h\in\Omega_{1}$.  Our description of the range will be based on this choice of $h$.
Define
$$
A_{\xi,g}=A := -h^{-1} Xh.
$$
Observe that since $h\in\Omega_1$, then $h^{-1}=\bar{h}\in\Omega_{-1}$. Also $Xh=\eta_{+}h+\eta_{-}h\in\Omega_{2}\oplus\Omega_{0}$ which implies that $A\in\Omega_{1}\oplus\Omega_{-1}$. It follows that $A$ is the restriction to $SM$ of a purely imaginary 1-form on $M$, hence we have a unitary connection (see Section \ref{sec:connections}).

First we will describe the range of the geodesic ray transform $I$ restricted to $\Omega_{m}$:
$$I_{m}:=I|_{\Omega_{m}}:\Omega_{m}\to C^{\infty}(\partial_{+}(SM),\C).$$
Observe that if $u$ solves the transport equation $Xu=-f$ with $u|_{\partial_{-}(SM)}=0$, then
$h^{-m}u$ solves $(X-mA)(h^{-m}u)=-h^{-m}f$ and $h^{-m}u|_{\partial_{-}(SM)}=0$.
Also note that $h^{-m}f\in\Omega_{0}$.
Thus
\begin{equation*}\label{eq:relation}
I_{-mA}^{0}(h^{-m}f)=\left(h^{-m}|_{\partial_{+}(SM)}\right)I_{m}(f)
\end{equation*}
where the left hand side is the {\it attenuated ray transform} of the unitary connection $-mA$. Attenuated transforms will be described in more detail in the next section, but the upshot is that we can prove a theorem similar to Theorem \ref{thm:range} but introducing this time a unitary connection as attenuation. Putting everything together one obtains a description of the range for $I_{m}$ as follows. Let
\[Q_{m}w(x,v):=\left\{\begin{array}{ll}
w(x,v)&\mbox{\rm if}\;(x,v)\in\partial_{+}(SM)\\
(e^{-m\int_{0}^{\tau(x,v)}A(\phi_{t}(x,v))\,dt}w)\circ\alpha(x,v)&\mbox{\rm if}\;(x,v)\in\partial_{-}(SM)\\
\end{array}\right.\] 
and
\[B_{m}g:=[e^{m\int_{0}^{\tau(x,v)}A(\phi_{t}(x,v))\,dt}(g\circ\alpha)-g]|_{\partial_{+}(SM)}.\]
In other words:
\[Q_{m}w(x,v)=\left\{\begin{array}{ll}
w(x,v)&\mbox{\rm if}\;(x,v)\in\partial_{+}(SM)\\
(e^{-mI_{1}(A)}w)\circ\alpha(x,v)&\mbox{\rm if}\;(x,v)\in\partial_{-}(SM)\\
\end{array}\right.\] 
and
\[B_{m}g=[e^{mI_{1}(A)}(g\circ\alpha)-g]|_{\partial_{+}(SM)}.\]
We define
\[P_{m,-}:=B_{m}H_{-}Q_{m}.\]
Then:

\begin{Theorem}[\cite{PSU3}] \label{thm:im} Let $(M,g)$ be a simple surface. Then a function $u\in C^{\infty}(\partial_{+}(SM),\C)$ belongs to the range of $I_{m}$ if and only if $u=\left(h^{m}|_{\partial_{+}(SM)}\right)P_{m,-}w$ for $w\in \mathcal S_{m}^{\infty}(\partial_{+}(SM),\C)$, where this last space denotes the set of all smooth $w$ such that $Q_{m}w$ is smooth.

\end{Theorem}

Suppose now $F$ is a complex-valued symmetric tensor of order $m$ and we denote its restriction to $SM$ by $f$. Recall from \cite[Section 2]{PSU} that there is a 1-1 correspondence between 
complex-valued symmetric tensors of order $m$ and functions in $SM$ of the form
$f=\sum_{k=-m}^{m}f_k$ where $f_k\in\Omega_k$ and $f_k=0$ for all $k$ odd (resp. even) if
$m$ is even (resp. odd).

Since
\[I(f)=\sum_{k=-m}^{m}I_{k}(f_{k})\]
we deduce  directly from Theorem \ref{thm:im} the following.

\begin{Theorem} Let $(M,g)$ be a simple surface. If $m=2l$ is even, a function 
$u\in C^{\infty}(\partial_{+}(SM),\C)$ belongs to the range of the ray transform acting on complex-valued  symmetric $m$-tensors if and only if there are $w_{2k}\in \mathcal S_{2k}^{\infty}(\partial_{+}(SM),\C)$
such that
\[u=\sum_{k=-l}^{l}\left(h^{2k}|_{\partial_{+}(SM)}\right)P_{2k,-}w_{2k}.\]
Similarly, if $m=2l+1$ is odd, a function 
$u\in C^{\infty}(\partial_{+}(SM),\C)$ belongs to the range of the ray transform acting on complex-valued  symmetric $m$-tensors if and only if there are $w_{2k+1}\in \mathcal S_{2k+1}^{\infty}(\partial_{+}(SM),\C)$
such that
\[u=\sum_{k=-l-1}^{l}\left(h^{2k+1}|_{\partial_{+}(SM)}\right)P_{2k+1,-}w_{2k+1}.\]
\label{thm:rangetensor}

\end{Theorem}

\section{Attenuated ray transform for unitary connections}
\label{sec:connections}

In this section we describe in detail certain injectivity results for the attenuated ray transform of a unitary connection \cite{PSU2}. We saw the appearance of the attenuated ray transform in the last section when we discussed the range of the (unattenuated) ray transform on tensors of any order. We also saw how
useful was for the tensor tomography problem to introduce a connection to gain positivity
in the Pestov identity. Here we take a closer and more systematic look. We motivate
this section by discussing first another natural inverse problem: determine a unitary connection
from its scattering relation, that is, parallel transport along geodesics between boundary points.
Our results are for simple surfaces, but the definitions can be given in the context of non-trapping manifolds $(M,g)$ with strictly convex boundary.

Suppose $E\to M$ is a Hermitian vector bundle of rank $n$ over $M$ and $\nabla$ is a unitary connection on $E$. Associated with $\nabla$ there is the following additional piece of scattering data: given $(x,v)\in \partial_{+}(SM)$, let
$P(x,v)=P_{\nabla}(x,v):E(x)\to E(\pi\circ\alpha(x,v))$ denote the parallel transport along
the geodesic $\gamma(t,x,v)$. This map is a linear isometry and the main
inverse problem we wish to discuss here is the following:

\medskip

\noindent{\bf Question.} Does $P$ determine $\nabla$?

\medskip

The first observation is that the problem has a natural gauge equivalence.
Let $\psi$ be a gauge transformation, that is, a smooth section of the
bundle of automorphisms $\mbox{\rm Aut} E$. The set of all these sections naturally forms a group (known as the gauge group) which acts on 
the space of unitary connections by the rule
\[(\psi^*\nabla) s:=\psi \nabla(\psi^{-1}s)\]
where $s$ is any smooth section of $E$.
If in addition $\psi|_{\partial M}=\id$, then it is a simple exercise to check that
\[P_{\nabla}=P_{\psi^*\nabla}.\]

Thus we can rephrase the question above more precisely as follows:

\medskip

\noindent{\bf Question I.}   Let $\nabla_1$ and $\nabla_2$ be two
unitary connections with $P_{\nabla_{1}}=P_{\nabla_{2}}$. Does there
exist a gauge transformation $\psi$ with $\psi|_{\partial M}=\id$
and $\psi^*\nabla_{1}=\nabla_{2}$?

\medskip

It is easy to see from the definition that a simple manifold must be diffeomorphic to a ball in $\mathbb R^n$. Therefore any bundle over such $M$
is necessarily trivial and from now on we shall assume that $E=M\times\C^n$.

Question I arises naturally when considering the hyperbolic Dirichlet-to-Neumann map associated to the Schr\"odinger equation with a connection. It was shown in \cite{FU} that when the metric is Euclidean, the scattering data for a connection can be determined from the hyperbolic Dirichlet-to-Neumann map. A similar result holds true on simple Riemannian manifolds: a combination of the methods in \cite{FU} and \cite{U2} shows that the hyperbolic Dirichlet-to-Neumann map for a connection determines the scattering data $P_{\nabla}$.

\bigskip

\noindent {\bf Elementary background on connections.} \ 
Consider the trivial bundle $M \times \C^n$. For us a connection $A$ will be a complex $n\times n$ matrix whose
entries are smooth 1-forms on $M$. Another way to think of $A$ is to regard
it as a smooth map $A:TM\to \C^{n\times n}$ which is linear in $v\in T_{x}M$ for
each $x\in M$.

Very often in physics and geometry one considers {\it unitary} or {\it Hermitian} connections. This means that the range of $A$ is restricted to skew-Hermitian matrices. In other words, if we denote by $\mathfrak{u}(n)$ the Lie algebra of the unitary group $U(n)$, we have a smooth map
$A:TM\to \mathfrak{u}(n)$ which is linear in the velocities.
There is yet another equivalent way to phrase this. The connection $A$ induces
a covariant derivative $d_{A}$ on sections $s\in C^{\infty}(M,\C^n)$ by setting
$d_{A}s=ds+As$. Then $A$ being Hermitian or unitary is equivalent to requiring compatibility with the standard Hermitian inner product of $\C^n$ in the sense that
\[d\langle s_{1},s_{2}\rangle=\langle d_{A}s_{1},s_{2}\rangle+\langle s_{1},d_{A}s_{2}\rangle\]
for any pair of functions $s_{1},s_{2}$.

Given two unitary connections $A$ and $B$ we shall say that $A$ and $B$
are gauge equivalent if there exists a smooth map $u:M\to U(n)$ such that
\begin{equation}
B=u^{-1}du+u^{-1}Au.
\label{eq:1}
\end{equation}
It is easy to check that this definition coincides with the
one given in the previous section if we set $\psi=u^{-1}$.

The {\it curvature} of the connection is the 2-form $F_{A}$ with values
in $\mathfrak u(n)$ given by
\[F_{A}:=dA+A\wedge A.\]
If $A$ and $B$ are related by (\ref{eq:1}) then:
\[F_{B}=u^{-1}\,F_{A}\,u.\]
Given a smooth curve $\gamma:[a,b]\to M$, the {\it parallel transport} along
$\gamma$ is obtained by solving the linear differential equation in $\C^n$:
\begin{equation}
\left\{\begin{array}{ll}
\dot{s}+A(\gamma(t),\dot{\gamma}(t))s=0,\\
s(a)=w\in \C^n.\\
\end{array}\right.
\label{eq:2}
\end{equation}
The isometry $P_{A}(\gamma):\C^n\to\C^n$ is defined as $P_{A}(\gamma)(w):=s(b)$.
We may also consider the fundamental unitary matrix solution $U:[a,b]\to U(n)$ of (\ref{eq:2}).
It solves
\begin{equation}
\left\{\begin{array}{ll}
\dot{U}+A(\gamma(t),\dot{\gamma}(t))U=0,\\
U(a)=\id.\\
\end{array}\right.
\label{eq:3}
\end{equation}
Clearly $P_{A}(\gamma)(w)=U(b)w$.

\bigskip

\noindent {\bf The transport equation and the attenuated ray transform.} \ 
Consider now the case of a compact simple Riemannian manifold.
We would like to pack the information provided by (\ref{eq:3})
along every geodesic into one PDE in $SM$. For this
we consider the vector field $X$ associated with the geodesic flow
$\phi_t$ and we look at the unique solution $U_{A}:SM\to U(n)$ of

\begin{equation}
\left\{\begin{array}{ll}
X(U_{A})+A(x,v)U_{A}=0,\;(x,v)\in SM\\
U_{A}|_{\partial_{+}(SM)}=\id.\\
\end{array}\right.
\label{eq:4}
\end{equation}
The scattering data of the connection $A$ is now the map $C_{A}:\partial_{-}(SM)\to U(n)$ defined as $C_{A}:=U_{A}|_{\partial_{-}(SM)}$.

We can now rephrase Question I as follows:

\medskip

\noindent{\bf Question I.} Let $A$ and $B$ be two
unitary connections with $C_{A}=C_{B}$. Does there
exist a smooth map $U:M\to U(n)$ with $U|_{\partial M}=\id$
and $B=U^{-1}dU+U^{-1}AU$?

\medskip

Suppose $C_A=C_B$ and define $U:=U_{A}(U_{B})^{-1}:SM\to U(n)$. One easily checks that
$U$ satisfies:
\[\left\{\begin{array}{ll}
XU + AU-UB= 0,\\
U|_{\partial(SM)}=\id.\\
\end{array}\right.\]
If we show that $U$ is in fact smooth {\it and} it only depends
on the base point $x\in M$ we would have an answer to Question I , since
the equation above reduces to $dU+AU-UB=0$ and $U|_{\partial M}=\id$ which is exactly gauge equivalence. Showing that $U$ only depends on $x$ is not an easy task and it often is the crux of the matter in these type of problems.
To tackle this issue we will rephrase the problem in terms of an {\it
attenuated ray transform}.

Consider $W:=U-\id:SM\to \C^{n \times n}$, where as before $\C^{n \times n}$ stands for the set of all
$n\times n$ complex matrices. Clearly $W$ satisfies
\begin{align}
&XW+AW-WB= B - A, \,\label{eq:uno}\\
&W|_{\partial(SM)}=0.\label{eq:dos}
\end{align}

We introduce a new connection $\hat{A}$ on the trivial bundle $M\times \C^{n \times n}$ as follows: given a matrix $R \in \C^{n \times n}$ we define
$\hat{A}(R):=AR-RB$. One easily checks that $\hat{A}$ is Hermitian
if $A$ and $B$ are. Then equations (\ref{eq:uno}) and (\ref{eq:dos})
are of the form:

\[
\left\{\begin{array}{ll}
Xu+Au=-f,\\
u|_{\partial(SM)}=0.\\
\end{array}\right.\]
where $A$ is a unitary connection, $f:SM\to\C^N$ is a smooth
function linear in the velocities, $u:SM\to\C^N$ is a function
that we would like to prove smooth and only dependent on $x\in M$ 
and $N=n\times n$.
As we will see shortly this amounts to understanding
which functions $f$ linear in the velocities are in the kernel
of the attenuated ray transform of the connection $A$.

First recall that in the scalar case, the attenuated ray transform $I_a f$ of a function $f \in C^{\infty}(SM,\C)$ with attenuation coefficient $a \in C^{\infty}(SM,\C)$ can be defined as the integral 
$$
I_a f(x,v) := \int_0^{\tau(x,v)} f(\phi_t(x,v)) \text{exp}\left[ \int_0^t a(\phi_s(x,v)) \,ds \right] dt, \quad (x,v) \in \partial_+(SM).
$$
Alternatively, we may set $I_a f := u|_{\partial_+ (SM)}$ where $u$ is the unique solution of the {\it transport equation} 
$$
Xu + au = -f \ \ \text{in $SM$}, \quad u|_{\partial_-(SM)} = 0.
$$

The last definition generalizes without difficulty to the case of connections. Assume that $A$ is a unitary connection and let $f \in C^{\infty}(SM,\C^n)$ be a vector valued function. Consider the following transport equation for $u: SM \to \C^n$, 
$$
Xu + Au = -f \ \ \text{in $SM$}, \quad u|_{\partial_-(SM)} = 0.
$$
On a fixed geodesic the transport equation becomes a linear ODE with zero initial condition, and therefore this equation has a unique solution denoted by $u^f$.

\begin{Definition}
The attenuated ray transform of $f \in C^{\infty}(SM,\C^n)$ is given by 
$$
I_{A} f := u^f|_{\partial_+(SM)}.
$$
\end{Definition}

We note that $I_{A}$ acting on sums of $0$-forms and $1$-forms always has a nontrivial kernel, since 
$$
I_{A}(dp+Ap) = 0 \text{ for any $p \in C^{\infty}(M,\C^n)$ with $p|_{\partial M} = 0$}.
$$
Thus from the ray transform $I_{A} f$ one only expects to recover $f$ up to an element having this form.

The transform $I_{A}$ also has an integral representation. Consider the unique matrix solution
$U_{A}:SM\to U(n)$ from above. Then it is easy to check that
$$
I_{A} f(x,v) = \int_{0}^{\tau(x,v)} U_{A}^{-1}(\phi_{t}(x,v))f(\phi_{t}(x,v))\,dt.
$$

We are now in a position to state the next main question:

\medskip

\noindent{\bf Question II.} (Kernel of $I_{A}$) Let $(M,g)$ be a compact simple
 Riemannian manifold and let $A$ be a unitary connection. Assume that $f:SM\to\C^n$ is a smooth function of the form
$F(x)+\alpha_{j}(x)v^j$, where $F:M\to\C^n$ is a smooth function
and $\alpha$ is a $\C^n$-valued 1-form. If $I_{A}(f)=0$, is it true
that $F=0$ and $\alpha=d_{A}p=dp+Ap$, where $p:M\to\C^n$ is a smooth
function with $p|_{\partial M}=0$?
\medskip

As explained above a positive answer to Question II gives a positive answer
to Question I. The next recent result provides a full answer to
Question II in the two-dimensional case:

\begin{Theorem}\cite{PSU2} Let $M$ be a compact simple surface. Assume that $f:SM\to\C^n$ is a smooth function of the form
$F(x)+\alpha_{j}(x)v^j$, where $F:M\to\C^n$ is a smooth function
and $\alpha$ is a $\C^n$-valued 1-form. Let also $A: TM \to \mathfrak{u}(n)$ be a unitary connection. If $I_{A}(f)=0$, then
$F =0$ and $\alpha=d_{A}p$, where $p:M\to\C^n$ is a smooth
function with $p|_{\partial M}=0$.
\label{thm:injective_higgs}
\end{Theorem}

Let us explicitly state the positive answer to Question I in the case
of surfaces:

\begin{Theorem}\cite{PSU2} Assume $M$ is a compact simple surface and let
$A$ and $B$ be two unitary connections.
Then $C_{A}=C_{B}$ implies that there exists a smooth
$U:M\to U(n)$ such that $U|_{\partial M}=\id$ and
$B=U^{-1}dU+U^{-1}AU$.
\label{thm:inverse}
\end{Theorem}

The proof of Theorem \ref{thm:injective_higgs} is based on the ideas explained in Section \ref{sec:firstproof}. One introduces a suitable additional attenuation (twists with a positive line bundle) which adds positivity to the Pestov identity with a connection and then gauges the twist away via the
key Theorem \ref{thm_holomorphic_integrating_factors}.

In the case of Euclidean space with the Euclidean metric the attenuated ray transform is the basis of the medical imaging technology of SPECT and has been extensively studied, see \cite{finch} for a review.
We remark that in connection with injectivity results for ray transforms, there is great interest in reconstruction procedures and inversion formulas. For the attenuated ray transform in $\re^2$ with Euclidean metric and scalar attenuation function, an explicit inversion formula was proved by R. Novikov \cite{No_inversion}. A related formula also including $1$-form attenuations appears in \cite{BS}, inversion formulas for matrix attenuations in Euclidean space are given in \cite{E, No}, and the case of hyperbolic space ${\mathbb H}^2$ is considered in \cite{Bal}.

Various versions of Theorem \ref{thm:inverse} have
been proved in the literature. Sharafutdinov \cite{Sha} proves the theorem
assuming that the connections are $C^1$ close to another connection
with small curvature (but in any dimension).
In the case of domains in the Euclidean plane the theorem was proved by Finch and Uhlmann \cite{FU}
assuming that the connections have small curvature and by
G. Eskin \cite{E} in general.
R. Novikov \cite{No} considers the case of connections which
are not compactly supported (but with suitable decay conditions at infinity) and establishes local uniqueness of the trivial connection and gives examples in which global uniqueness fails (existence of ``ghosts"). 

For more on inverse problems for connections we refer to \cite{Paternain_survey}.

\section{Anosov manifolds}
\label{sec:anosov}

There are versions of the ideas in the previous sections in the context of closed manifolds. The first requirement is to have a notion that
replaces the concept of simple manifold. It is easy to motivate this as follows. Simple manifolds have two characteristic properties:
they have no conjugate points and and they are open in the $C^2$-topology of metrics.  Recall that a Riemannian manifold is said to have no conjugate points if any two points in the universal covering are joined by a unique geodesic.
Hence it seems natural to seek an analogue by requiring that the metric is a $C^2$-interior point among the set of all metric without
conjugate points. 

\begin{Definition} A closed Riemannian manifold $(M,g)$ is said to be Anosov if $g$ belongs to the $C^2$-interior of the set of metrics
without conjugate points.
\end{Definition}

It turns out that the name ``Anosov" is completely justified: $(M,g)$ is Anosov if and only if the geodesic flow of $g$ is Anosov in the sense of
Dynamical Systems \cite{Ru}.  We will not give here the definition of an Anosov flow since it will not be explicitly needed and instead we refer the reader to \cite{KH}. 

From our definition it is clear that negatively curved manifolds are Anosov and that there are no Anosov metrics on tori since the only metrics
without conjugate points on tori must be flat \cite{BI}.

The notion of ``$I_{m}$ is $s$-injective" makes sense for closed manifolds as follows:

\begin{Definition} We say that $I_{m}$ is $s$-injective if given any symmetric $m$-tensor $f$ such that
\[\int_{0}^{T}f_{m}(\gamma(t),\dot{\gamma}(t))\,dt=0\]
for every unit speed closed geodesic $\gamma:[0,T]\to M$, then $f$ is potential, i.e., there exists an $(m-1)$-symmetric tensor $h$ 
such that $f=dh$.
\end{Definition}

The tensor tomography problem for an Anosov manifold consists in proving that $I_{m}$ is $s$-injective for any $m$. There are numerous motivations for this, but perhaps the most notorious one is that of spectral rigidity which involves $I_{2}$. In \cite{GK} Guillemin and Kazhdan proved that
if $(M,g)$ is an Anosov manifold such that $I_{2}$ is $s$-injective then $(M,g)$ is spectrally rigid. This means that if $(g_s)$ is a smooth family of Riemannian metrics on $M$ for $s \in (-\varepsilon,\varepsilon)$ such that $g_0 = g$ and the spectra of $-\Delta_{g_s}$ coincide up to multiplicity,
$$
\text{Spec}(-\Delta_{g_s}) = \text{Spec}(-\Delta_{g_0}), \quad s \in (-\varepsilon,\varepsilon),
$$
then there exists a family of diffeomorphisms $\psi_s: M \to M$ with $\psi_0 = \text{Id}$ and 
$$
g_s = \psi_s^* g_0.
$$

Let us summarize what is known about the tensor tomography problem on an Anosov manifold.

\begin{itemize}
\item $I_{0}$ and $I_{1}$ are $s$-injective \cite{DS};
\item $I_{2}$ is $s$-injective for surfaces \cite{PSU4} ;
\item $I_{m}$ is $s$-injective for all $m$ for non-positively curved manifolds \cite{CS}.
\end{itemize}

\section{Open problems}
\label{sec:openproblems}

In this section we mention some open problems related to tensor tomography.

\begin{enumerate}
\item[1.]

In the two dimensional case there is by now, as surveyed in this paper, a rather good understanding of the injectivity and range of the geodesic ray transform on tensor fields for simple manifolds. Important questions remaining are inversion formulas or reconstruction procedures of the solenoidal part of the tensor field from its geodesic ray transform. Certain inversion formulas were given in \cite{PU2}, \cite{Kr2} for the case of constant curvature and close to constant curvature.
\item[2.]
In the case where $\dim(M)Ê\geq 3$ it is not known whether $I_m$ is $s$-injective for a general simple manifold. This is known for $I_0$ and $I_1$, but even the case of $I_2$ is unknown at present.
\item[3.]
Support type theorems for the geodesic ray transform, where a tensor field is determined locally from its line integrals in a certain neighborhood, are known for the case of real analytic simple manifolds for $I_m$, $m=0,1,2$
\cite{K,KS}. Is it possible to extend these results to all simple manifolds? This has been done for $I_0$ in three dimensions or higher \cite{UV}.
\item[4.]
The study of $s$-injectivity of the geodesic ray transform for non-simple
manifolds is an important problem for which not much is known. Certain results are given in \cite{D}, \cite{Sh1/2}, \cite{Sh3/2}. A microlocal analysis of $I_0$ when the exponential map has fold type singularities was done in \cite{SU5}. Injectivity, stability and reconstruction were proven for $I_0$ in the case of three dimensions or higher when the manifold can be foliated by strictly convex hypersurfaces \cite{UV}. This allows for conjugate points. The $s$-injectivity of $I_2$ was analyzed in \cite{SU4} for a class of non-simple manifolds. However, the question if $I_0$ is injective on a compact non-trapping manifold with strictly convex boundary is open.
\item[5.]
The attenuated ray transform for an unitary connection on simple surfaces and Anosov surfaces has been extensively studied in \cite{P,PSU2,SaU}. It would be interesting to extend  the results to the case of a non-unitary connection.
\item[6.]
For closed Anosov surfaces it is known that $I_{m}$  is $s$-injective for $m=0,1,2$ . Is it true for all $m$?
Also, is $I_{2}$ $s$-injective for Anosov manifolds of dimension $\geq 3$?
\item[7.]
Finally, it would be natural to extend all this theory to more general classes of curves. By this we mean replacing geodesics by other natural set of curves like magnetic geodesics or geodesics of affine connections with torsion (thermostats). Concerning magnetic geodesics, the tensor tomography problem in $2D$ is solved in \cite{Ai} using the ideas presented here and the results in \cite{DPSU}. See also \cite{AD}.
\end{enumerate}

\end{document}